\documentclass[a4paper,oneside,11pt]{article}

\usepackage{geometry}                % See geometry.pdf to learn the layout options. There are lots.
\usepackage{graphicx,color}
\usepackage{amsmath,amsfonts,amssymb,amsthm}
\usepackage{enumerate}
\usepackage[utf8]{inputenc} 
\usepackage[T1]{fontenc} 
 
\usepackage[english]{babel} 

\title{Automorphisms of graphs of cyclic splittings of free groups}
\author{Camille Horbez  and Richard D. Wade}

\usepackage[all]{xy}

\usepackage{bbm}
\begin{document}
\maketitle
%\tableofcontents
\newtheorem*{thma}{Theorem A}
\newtheorem{de}{Definition} [section]
\newtheorem{theo}[de]{Theorem} 
\newtheorem{prop}[de]{Proposition}
\newtheorem{lemma}[de]{Lemma}
\newtheorem{cor}[de]{Corollary}
\newtheorem{propd}[de]{Proposition-Definition}
\theoremstyle{remark}
\newtheorem{rk}[de]{Remark}

\begin{abstract}
We prove that any isometry of the graph of cyclic splittings of a finitely generated free group $F_N$ of rank $N\ge 3$ is induced by an outer automorphism of $F_N$. The same statement also applies to the graphs of maximally-cyclic splittings, and of very small splittings. \end{abstract}

\section*{Introduction}

The study of the outer automorphism group of a finitely generated free group has benefited greatly from analogies with mapping class groups of compact surfaces. In recent years, research has concentrated on understanding the geometry of proposed analogues of the curve graph of a surface, one main question being that of finding natural hyperbolic $\text{Out}(F_N)$-graphs. Among these stand the free factor graph, the free splitting graph and the cyclic splitting graph. We refer the reader to \cite{BF12,BF13,BR13,Ham12,HM12,HM14,HH13,KR12,Man12} for various results about these graphs.

In the context of closed orientable surfaces, Royden proved that (except in a finite number of sporadic cases) the group of biholomorphisms of Teichmüller space, as well as its group of isometries with respect to the Teichmüller metric, coincides with the extended mapping class group of the surface \cite{Roy71}. His result was extended by Earle and Kra to the case of punctured surfaces \cite{EK74-2}. Similar results are also known to hold for the Weil--Petersson metric \cite{MW02,BM07} or Thurston's asymmetric metric \cite{Wal11}. Ivanov, Korkmaz and Luo's rigidity result for the curve graph of a (nonsporadic) compact surface \cite{Iva97,Kor99,Luo00} states that the group of simplicial isometries of this graph also coincides with the extended mapping class group. The reader is referred to \cite{McP12} for a list of various rigidity results for simplicial actions of mapping class groups. 

There are some known analogous results for the group $\text{Out}(F_N)$. Bridson and Vogtmann \cite{BV01} first proved that when $N\ge 3$, the group $\text{Out}(F_N)$ is the group of simplicial automorphisms of the spine of Outer space, and Francaviglia and Martino then used this to show that the group of isometries of Outer space with the Lipschitz metric \cite{FM12} is also equal to $\text{Out}(F_N)$ when $N\ge 3$, and to $PSL(2,\mathbb{Z})$ when $N=2$ (another approach to this result, based on a study of the metric completion of Outer space, is due to Algom-Kfir \cite{AK12}). Building on Bridson and Vogtmann's result, Aramayona and Souto proved that for $N\ge 3$, the group $\text{Out}(F_N)$ is also the group of simplicial isometries of the free splitting graph \cite{AS11}. This paper is concerned with a rigidity question concerning the isometry group of Mann's cyclic splitting graph \cite{Man12}, and a pair of its close relatives.

A \emph{splitting} of $F_N$ is a simplicial tree on which $F_N$ acts by simplicial automorphisms with no proper invariant subtree.  Two splittings are \emph{equivalent} if there exists an $F_N$-equivariant homeomorphism between them. Given two splittings $T$ and $T'$ of $F_N$, we say that $T$ is a \emph{refinement} of $T'$ if $T'$ is obtained by equivariantly collapsing some of the edges in $T$ to points. Let $FZ_N$ denote the graph of cyclic splittings of $F_N$, i.e. the graph whose vertices are the equivalence classes of splittings of $F_N$ whose edge stabilizers are cyclic (possibly trivial) subgroups of $F_N$. Two such splittings are joined by an edge if one is a proper refinement of the other. Let $FZ_N^{max}$ be the graph of maximally-cyclic splittings of $F_N$, which is defined in the same way with the extra assumption that edge stabilizers in the splittings we consider are closed under taking roots. We will also consider the graph $VS_N$ of very small splittings of $F_N$. This consists of  maximally-cyclic splittings that are additionally required to have trivial tripod stabilizers. The graphs $FZ_N$, $FZ_N^{max}$ and $VS_N$, equipped with the path metric, come with right isometric actions of the group $\text{Out}(F_N)$ of outer automorphisms of $F_N$ (an automorphism $\phi \in \text{Aut}(F_n)$ induces an isometry of each complex by precomposing each $F_N$--action on a tree by $\phi$, and this isometry depends only on the outer automorphism class of $\phi$). The goal of this paper is to show the following rigidity result, which states that every isometry of either $FZ_N$, $FZ_N^{max}$ or $VS_N$ is in fact induced by the action of an element of $\text{Out}(F_N)$. 

\begin{thma} \label{automorphism-group} 
For all $N\ge 3$, the natural maps from $\text{Out}(F_N)$ to the isometry groups of $FZ_N$, $FZ_N^{max}$ and $VS_N$ are isomorphisms. 
\end{thma}

A description of all maximally-cyclic splittings of $F_2$ can be found in \cite{CV91}. The complex $FZ_2^{max}$ turns out to be isomorphic to the Farey graph with depth two dead ends and "fins" attached. In particular, its automorphism group is isomorphic to $PSL(2,\mathbb{Z})$.

We may then assume that $N \geq 3$. Let $\mathcal{G}$ be one of the graphs $FZ_N$, $FZ_N^{max}$ or $VS_N$. The free splitting graph $FS_N$, whose vertices are the equivalence classes of splittings of $F_N$ with trivial edge stabilizers, sits as a subcomplex inside $\mathcal{G}$. Our proof of Theorem~A relies on Aramayona and Souto's rigidity statement  \cite{AS11} for the free splitting graph: we first prove that any isometry of $\mathcal{G}$ preserves the subgraph $FS_N$ setwise, and then show that any isometry of $\mathcal{G}$ which restricts to the identity on $FS_N$ is actually the identity map.

The restriction to the set of very small splittings is a natural one -- these are exactly the cyclic splittings which arise in the boundary of Outer space (see \cite{CL95} where the notion of a very small splitting was introduced for the first time). The natural inclusions of $FZ_N^{max}$ and $VS_N$ into $FS_N$ may not be quasi-isometries, however Mann's proof \cite{Man12} translates to show that $FZ_N^{max}$ and $VS_N$ are also hyperbolic.

We finally note that the problem of determining the group of simplicial isometries of the free factor graph is still open.

\section*{Acknowledgments}

This work started during the programme "The Geometry of Outer space: Investigated through its analogy with Teichmueller space" held at Aix-Marseille Université during Summer 2013. We are greatly indebted to the organizers of this event. We would also like to thank Brian Mann for inspiring conversations we had there.

\section{The structure of cyclic splittings of $F_N$}\label{s:vs}  

Recall that a \emph{splitting} of $F_N$ is a simplicial action of $F_N$ on a simplicial tree $T$ with no proper and nontrivial invariant subtrees. A splitting is \emph{cyclic} if every edge stabilizer is cyclic (possibly trivial), and is \emph{maximally-cyclic} if each nontrivial edge stabilizer is a maximally-cyclic subgroup of $F_N$ (i.e. edge stabilizers are closed under taking roots). When $\mathcal{G}$ is one of the graphs $FS_N$, $FZ_N$, $FZ_N^{max}$ or $VS_N$ defined in the introduction, we say that a splitting is \emph{$\mathcal{G}$-maximal} if it admits no nontrivial refinement in $\mathcal{G}$. 

Given an edge $e$ in an $F_N$-tree $T$ adjacent to a vertex $v$, we denote by $[e]$ the $G_v$-orbit of the edge $e$ and by $[G_e]$ the $G_v$-conjugacy class of its edge group. The following lemma is a version of a theorem by Shenitzer and Swarup \cite{She55,Swa86}, see also \cite{Sta91} or \cite[Lemma 4.1]{BF94}. This will turn out to be crucial in our proof of Theorem~A for understanding the structure of cyclic splittings of $F_N$. 

\begin{lemma} (Shenitzer \cite{She55}, Swarup \cite{Swa86}, Stallings \cite{Sta91}, Bestvina--Feighn \cite[Lemma 4.1] {BF94}) \label{unfold}
Let $T$ be a cyclic splitting of $F_N$ with a nontrivial edge stabilizer. Then there exists an edge $e$ with nontrivial stabilizer $G_e$ adjacent to a vertex $v$ such that:

\begin{enumerate}[$(\star)$] \item There is a decomposition $G_v=G_e\ast A$ such that if $e'$ is another edge adjacent to $v$, with $[e']\neq[e]$, then some representative of $[G_{e'}]$ is contained in $A$.\end{enumerate} 
\end{lemma}

We say that an edge $e$ satisfying the condition $(\star)$ is \emph{unfoldable} at the vertex $v$. The above lemma tells us that we can find a (possibly trivial) refinement $T'$ of $T$ by (equivariantly) replacing the vertex $v$ with the tree corresponding to the above splitting of $G_v$. In the new quotient graph $T'/F_N$, the vertex with stabilizer  $G_e$ will have valence 2, and we have the following lemma: 

\begin{lemma} \label{unfold2} 
If $T$ is a maximally-cyclic splitting of $F_N$ with some nontrivial edge stabilizer then there is a refinement $T'$ of $T$ such that the quotient graph of groups $T'/F_N$ has a vertex of valence two where one adjacent edge has a trivial stabilizer, and the other is nontrivial. In particular, if $T$ is $\mathcal{G}$-maximal, then in the quotient graph $T/F_N$ each unfoldable edge has a vertex of valence two such that the other adjacent edge at this vertex has a trivial stabilizer. 
\qed \end{lemma}

Let $v$ be a vertex of valence $2$ in $T'/F_N$ given by Lemma \ref{unfold2}, and $e_1$ (resp. $e_2$) be the edge adjacent to $v$ with trivial (resp. nontrivial) stabilizer. We can equivariantly collapse the edge $e_2$ in $T'$ to obtain a new splitting $T''$. We say that $T''$ is obtained from $T$ by \emph{unfolding} the edge $e$, and $T'$ is a \emph{partial unfolding}. In the opposite direction $T'$ is obtained from $T''$ by partially folding the orbit of the edge added to split $G_v$, and $T$ is obtained from $T''$ by fully folding this orbit. Note that the unfolding operation is, in general, far from unique; there may be many possible choices for the complementary free factor $A$. For example, all splittings of $F_3=\langle a,b,c\rangle$ of the form $\langle a,b\rangle\ast\langle c[a,b]^k\rangle$ with $k\in\mathbb{Z}$ are obtained by (fully) unfolding the splitting $\langle a,b\rangle\ast_{[a,b]}\langle c,[a,b]\rangle$.

\paragraph*{The structure of one-edge cyclic splittings of $F_N$.}

A splitting $T$ is a \emph{$k$-edge splitting} if there are $k$ orbits of edges in $T$ under the action of $F_N$. To illustrate Lemma \ref{unfold}, we give the following classification of one-edge cyclic splittings of $F_N$. Such a splitting has one of the following forms:

\begin{itemize}
\item a \emph{separating one-edge free splitting} $F_N=A\ast B$, where $A$ and $B$ are complementary proper free factors of $F_N$, or
\item a \emph{nonseparating one-edge free splitting} $F_N=C\ast$, where $C$ is a corank one free factor of $F_N$, or 
\item a \emph{separating one-edge $\mathbb{Z}$-splitting} $F_N=A\ast_{\langle w\rangle}(B\ast\langle w\rangle)$, where $A$ and $B$ are complementary proper free factors of $F_N$, and $w\in A$, or
\item a \emph{nonseparating one-edge $\mathbb{Z}$-splitting} $F_N=(C\ast\langle w^t\rangle)\ast_{\langle w\rangle}$, where $C$ is a corank one free factor of $F_N$, and $w\in C$, and $t$ denotes the stable letter.
\end{itemize}

\paragraph*{Two useful results.}

We say that two splittings are \emph{compatible} if they both can be obtained by equivariantly collapsing edges of a common tree. Scott and Swarup showed that a $k$-edge splitting is determined by its set of $k$ one-edge collapses. Furthermore, it is enough for a set of splittings to be pariwise compatible to find a common refinement of the whole collection. 

\begin{theo}(Scott--Swarup \cite[Theorem 2.5]{SS00}, see also Handel--Mosher \cite[Lemma 1.3]{HM12})\label{t:ss}
Any set $T_1,\ldots,T_k$ of distinct, pairwise-compatible, one-edge cyclic splittings of $F_N$ has a unique $k$-edge refinement. Any $k$-edge cyclic splitting of $F_N$ refines exactly $k$ distinct one-edge splittings.
\end{theo}

At times we will also need to know the existence of a uniform bound on the number of edges in maximally-cyclic splittings of $F_N$. This was shown by Bestvina and Feighn.

\begin{theo}(Bestvina--Feighn \cite{BF91})\label{bf}
There is a uniform bound (depending only on $N$) for the number of orbits of edges in a maximally-cyclic splitting of $F_N$.
\end{theo}

Notice however that there is no bound on the number of orbits of edges in an arbitrary cyclic splitting of $F_N$. For example, the group $F_2=\langle a,b\rangle$ has arbitrarily long splittings of the form $F_2=\langle a\rangle\ast_{\langle a^2\rangle}\langle a^2\rangle\ast_{\langle a^4\rangle}\langle a^4\rangle\ast_{\langle a^8\rangle}\dots\ast\langle b\rangle$. 

\section{One-edge splittings and $\mathcal{G}$-maximal splittings}\label{sec-max}

Let $N\ge 3$. Throughout the section, we denote by $\mathcal{G}$ one of the graphs $FS_N$, $FZ_N$, $FZ_N^{max}$ or $VS_N$. The goal of this section is to provide a characterization of one-edge and $\mathcal{G}$-maximal splittings in terms of their combinatorics in $\mathcal{G}$. 

\begin{lemma}\label{maximal-valence}
All $\mathcal{G}$-maximal splittings have finite valence in $\mathcal{G}$.
\end{lemma}

\begin{proof}
A $\mathcal{G}$-maximal splitting $T$ is not properly refined by any splitting in $\mathcal{G}$, and the number of splittings it properly refines is equal to the number of proper subsets of orbits of edges in $T$ (see Theorem \ref{t:ss}), which is finite.
\end{proof}

\begin{lemma}\label{one-edge-valence}
All one-edge cyclic splittings of $F_N$ are compatible with infinitely many maximally-cyclic splittings of $F_N$. In particular, all one-edge splittings have infinite valence in $\mathcal{G}$.
\end{lemma}

\begin{proof}
Let $T$ be a one-edge cyclic splitting of $F_N$. If $T$ is a free splitting, then as $N\ge 3$, one of the vertex groups of $T$ has rank at least $2$, and splitting this vertex group yields infinitely many distinct proper refinements of $T$ (in particular $T$ is compatible with infinitely many free splittings).

Suppose that $T$ is a separating one-edge $\mathbb{Z}$-splitting of the form $F_N=A\ast_{\langle w\rangle}(B\ast\langle w\rangle)$, where $A$ and $B$ are complementary proper free factors of $F_N$, and $w\in A$. Let $\{b_1,\dots,b_k\}$ be a free basis of $B$. Then the splittings $A\ast_{\langle w\rangle}\langle w\rangle\ast \langle b_1w^i,b_2,\dots,b_k\rangle$, where $i$ varies in $\mathbb{N}$, yield infinitely many distinct two-edge refinements of $T$ (in particular $T$ is compatible with infinitely many free splittings).

Finally, assume that $T$ is a nonseparating one-edge $\mathbb{Z}$-splitting of the form $F_N=(C\ast\langle w^t\rangle)\ast_{\langle w\rangle}$, where $C$ is a corank one free factor of $F_N$, and $w\in C$, and $t$ denotes the stable letter. Then for each $g\in C$ which is not a proper power, the splitting $F_N=(C\ast\langle g^{t^{-1}}\rangle)\ast_{\langle g\rangle}$ is compatible with $T$, and this yields infinitely many distinct two-edge refinements of $T$ (in particular $T$ is compatible with infinitely many maximally-cyclic splittings).
\end{proof}

\begin{lemma}\label{maximal-one-edge}
For all $T\in \mathcal{G}$, the following conditions are equivalent.

\begin{itemize}
\item The splitting $T$ is either a one-edge splitting or a $\mathcal{G}$-maximal splitting.
\item There exist splittings $T_1$ and $T_2$ of $F_N$ such that \begin{itemize} \item $d_{\mathcal{G}}(T,T_1)=d_{\mathcal{G}}(T,T_2)=1$, and \item $d_{\mathcal{G}}(T_1,T_2)=2$, and \item $T_1-T-T_2$ is the unique path of length $2$ joining $T_1$ to $T_2$ in $\mathcal{G}$.\end{itemize}
\end{itemize}
\end{lemma}

\begin{proof}
If $T$ is a $\mathcal{G}$-maximal splitting, then any $F_N$-equivariant partition of the set of edges of $T$ into two subsets $E_1$ and $E_2$ gives rise to distinct splittings $T_i$, obtained by equivariantly collapsing edges in $E_i$ to points, that satisfy the desired properties by Theorem \ref{t:ss}. If $T$ is a one-edge splitting in $\mathcal{G}$, Lemma \ref{one-edge-valence} shows that it is compatible with infinitely many distinct one-edge maximally-cyclic splittings of $F_N$. If all these splittings were pairwise compatible, Theorem \ref{t:ss} would enable us to construct maximally-cyclic splittings of $F_N$ with arbitrarily large numbers of orbits of edges. This would contradict Theorem \ref{bf}. Therefore, we can find a pair of two-edge proper refinements $T_1$ and $T_2$ of $T$ which are not compatible. Such trees satisfy the desired properties. 

Conversely, assume that there exist splittings $T_1$ and $T_2$ satisfying the conclusions of the lemma. 
\begin{itemize}
\item If $T_1$ is properly refined by $T$ and $T$ is properly refined by $T_2$, then $T_1$ is properly refined by $T_2$, so $d_{\mathcal{G}}(T_1,T_2)=1$, a contradiction. 

\item If $T_1$ and $T_2$ are both properly refined by $T$, then $T$ is $\mathcal{G}$-maximal, otherwise we could find a proper refinement $T'$ of $T$ in $\mathcal{G}$, and get two paths of length $2$ joining $T_1$ to $T_2$ in $\mathcal{G}$ (namely, the path going through $T$ and the path going through $T'$). 

\item If $T$ is properly refined by both $T_1$ and $T_2$, then $T$ is a one-edge splitting, otherwise we could find a splitting $T'$ in $\mathcal{G}$ that is properly refined by $T$, and get two different paths of length $2$ joining $T_1$ to $T_2$ as above.  
\end{itemize}
As $d_{\mathcal{G}}(T,T_1)=d_{\mathcal{G}}(T,T_2)=1$, up to exchanging $T_1$ and $T_2$ one of the above cases occurs, and the claim follows.
\end{proof}

As a consequence of Lemmas \ref{maximal-valence}, \ref{one-edge-valence} and \ref{maximal-one-edge}, we get the following result.

\begin{prop}\label{max-preserved}
Any isometry of $\mathcal{G}$ preserves the sets of one-edge splittings and of $\mathcal{G}$-maximal splittings (setwise).
\qed
\end{prop}

\begin{rk}\label{rk-maximal-one-edge}
A free splitting of $F_N$ is maximal among free splittings of $F_N$ if and only if it is maximal among cyclic splittings of $F_N$ (such a splitting has trivial vertex stabilizers, and hence cannot be refined by a splitting having nontrivial edge stabilizers). Hence we can talk about "maximal free splittings" without any ambiguity.
\end{rk}

\begin{rk}\label{edge-inv}
As the number of edges of $T/F_N$ is equal to the number of one-edge splittings adjacent to $T$ (Theorem \ref{t:ss}), any isometry of $\mathcal{G}$ preserves the number of edge orbits in each splitting. As a consequence, the property that $T'$ is obtained from $T$ by either collapsing or adding a fixed number of edges is preserved under an isometry of $\mathcal{G}$. 
\end{rk}

\section{Invariance of the free splitting graph}

The first step in our proof of Theorem~A is to show that an isometry of the graph $FZ_N$, $FZ_N^{max}$ or $VS_N$ preserves the subgraph $FS_N$ of free splittings of $F_N$ setwise. We do this by distinguishing maximal free splittings from maximal splittings which have at least one nontrivial edge stabilizer. We first look at the graph $FZ_N$. In this situation the argument is much simpler as all maximal splittings are free. 

\subsection{The case of $FZ_N$.}

\begin{prop}\label{max-free}
All $FZ_N$-maximal splittings are free splittings.
\end{prop}

\begin{proof}
Assume towards a contradiction that there exists a cyclic splitting $T$ of $F_N$ that is $FZ_N$-maximal, and has an edge with nontrivial stabilizer. Lemma \ref{unfold2} implies that $T$ contains a vertex $v$ with nontrivial stabilizer $G_v$, which projects to a vertex of valence $2$ in the quotient graph $T/F_N$, and is adjacent to an edge $e$ with trivial stabilizer in $T$. By taking a proper power of a generator, we can find $g\in G_v$ that generates a proper subgroup of $G_v$, and partially fold $e$ along $ge$ in an equivariant way, to obtain a proper refinement of $T$. This contradicts $FZ_N$-maximality of $T$. 
\end{proof} 

\begin{lemma}\label{adjacent-max}
A splitting of $F_N$ is a free splitting if and only if it is at distance at most one from a maximal free splitting in $FZ_N$.
\end{lemma}

\begin{proof}
A splitting having a nontrivial edge stabilizer cannot be refined by a free splitting. We need to show that any free splitting enlarges to a maximal one. This follows from the existence of a bound on the number of edges in a free splitting of $F_N$ (it is well known that maximal free splittings have $3N -3$ edges but an abstract bound also follows from Theorem~\ref{bf}).
\end{proof}

\begin{cor}\label{fz-fs}
Any isometry of $FZ_N$ preserves the subgraph $FS_N$ (setwise).
\end{cor}

\begin{proof}
Propositions \ref{max-preserved} and \ref{max-free} imply that any isometry of $FZ_N$ preserves the set of maximal free splittings of $F_N$. The claim then follows from Lemma \ref{adjacent-max}.
\end{proof}

Proposition \ref{max-free} is no longer true for the graphs $FZ_N^{max}$ and $VS_N$, so in these cases we need a more refined argument to distinguish maximal free splittings from other $\mathcal{G}$-maximal splittings. For example, the splitting displayed on Figure \ref{fig-vsmax} is both $VS_N$-maximal and $FZ_N^{max}$-maximal. However, it fails to be $FZ_N$-maximal (one can add an extra cyclic edge with stabilizer $(cd)^2$ into the middle of the splitting).

\begin{figure}
\begin{center}
\includegraphics{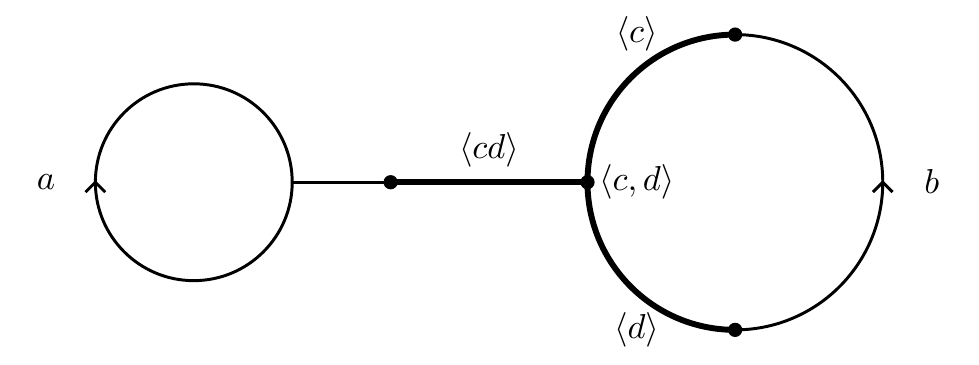}
\caption{A $VS_4$-maximal tree with nontrivial edge stabilizers.}
\label{fig-vsmax}
\end{center}
\end{figure}

\subsection{The case of $VS_N$.}

We saw in Section~\ref{sec-max} that $VS_N$-maximal splittings have finite valence in $VS_N$. We now distinguish maximal free splittings from other $VS_N$-maximal splittings, by using the following property: if a maximal splitting is free then collapsing one orbit of edges still gives a splitting of finite valence. In contrast, this is not satisfied by $VS_N$-maximal splittings that contain an edge with nontrivial stabilizer.

\begin{prop} \label{maximal-sphere}
Let $T$ be a maximal free splitting of $F_N$. Any splitting obtained by equivariantly collapsing one edge in $T$ has finite valence in $FZ_N^{max}$ (hence also in $VS_N$).
\end{prop}

\begin{proof}
It is a standard fact that all vertices in $T$ have valence $3$, and $T$ contains $3N-3$ orbits of edges. Equivariantly collapsing an edge that maps to a loop-edge in the quotient graph of groups creates a tree $T'$ whose quotient graph of groups has a vertex of valence $1$ with $\mathbb{Z}$ as its stabilizer. Any proper refinement of $T'$ must be obtained by inserting an edge attached to this vertex, and as we cannot split cyclic vertices along proper powers in $FZ_N^{max}$, one deduces that the tree $T$ is the only proper refinement of $T'$. Equivariantly collapsing an edge that does not map to a loop-edge creates a free splitting with trivial vertex groups where one vertex has valence $4$. There are exactly three ways of attaching back an edge to the quotient graph which are given by the possible pairings of edges at the vertex of valence $4$. Hence $T'$ has finite valence.  
\end{proof}

\begin{prop}\label{max-vsn}
Let $T$ be a $VS_N$-maximal splitting that has some nontrivial edge stabilizer. Then there exists an edge $e$ in $T$ such that equivariantly collapsing $e$ to a point yields a splitting of infinite valence in $VS_N$.\end{prop}

\begin{proof}
As $T$ is $VS_N$-maximal, Lemma~\ref{unfold2} shows the existence of an edge $e$ in $T$ adjacent to a vertex $v$ of valence $2$ (in the quotient graph $T/F_N$) such that its second adjacent edge $e'$ has a trivial stabilizer. Let $v'$ be the other vertex of $e$. 

If the stabilizer of $v'$ has rank at least $2$, let $T'$ be the splitting obtained by equivariantly collapsing $e$ to a point. Then $T'$ contains a vertex $v''$ (the image of $v'$ under the collapse map) whose stabilizer $G_{v''}$ has rank at least $2$, and which is adjacent to an edge $e'$ with trivial stabilizer. Hence $T'$ has infinite valence in $VS_N$, since one gets infinitely many distinct refinements of $T'$ by equivariantly folding an initial segment of $e'$ with an initial segment of $ge'$, for some $g\in G_{v''}$ which is not a proper power (and which can be chosen so as not to create any tripod stabilizer).

If the stabilizer of $v'$ is cyclic, minimality of the $F_N$-action, together with the fact that edge stabilizers are not proper powers, prevents $v'$ from having valence $1$ in the quotient graph $T/F_N$.  As $G_{v'}$ is cyclic, edge stabilizers are not proper powers, and $e$ is not contained in a tripod stabilizer, one of the adjacent edges of $v'$ has a trivial stabilizer. Let $T'$ be the tree obtained from $T$ by equivariantly collapsing $e$ to a point. Then $T'$ contains a vertex $v''$ with nontrivial cyclic stabilizer (a generator of which we denote by $t$), adjacent to two edges (denoted by $e_1$ and $e_2$) with trivial stabilizers that have initial segments which are in distinct $F_N$-orbits. Hence $T'$ has infinite valence in $VS_N$, since one gets infinitely many distinct refinements $T'_k$ of $T'$ by equivariantly folding an initial segment of $e_1$ with an initial segment of $t^ke_2$, for $k$ varying in $\mathbb{Z}$. To check that the splittings $T'_k$ are pairwise distinct, let $g\in F_N$ be an element whose axis in $T'$ crosses the turn $(e_1,t^ke_2)$ (this exists by minimality of the $F_N$-action on $T'$), and let $h\in F_N$ be an element whose axis crosses the turn $(te_1,t^2e_1)$. Then for all $l\in\mathbb{N}$, the axes of $g$ and $h$ meet in a single point in $T'_l$, except precisely when $k=l$, in which case they are disjoint.
\end{proof}

\begin{cor}\label{vs-fs}
Any isometry of $VS_N$ preserves the subgraph $FS_N$ (setwise).
\end{cor}

\begin{proof}
Let $f$ be an isometry of $VS_N$. Proposition \ref{max-preserved} shows that $f$ preserves the set of $VS_N$-maximal splittings. Every one-edge collapse of a maximal free splitting has finite valence (Proposition \ref{maximal-sphere}), whereas every non-free $VS_N$-maximal splitting has a one-edge collapse with infinite valence (Proposition \ref{max-vsn}), therefore $f$ preserves the set of maximal free splittings (Remark \ref{edge-inv} tells us that the property of being a one-edge collapse of an adjacent vertex is preserved under an isometry). As free splittings are characterized by being at distance at most one from a maximal free splitting (Lemma \ref{adjacent-max}), the map $f$ preserves $FS_N$.
\end{proof}

Proposition \ref{max-vsn} fails to be true for the graph $FZ_N^{max}$, as illustrated by the $FZ_4^{max}$-splitting displayed in Figure \ref{fig-max-fz}. Again we will have to refine the argument to distinguish maximal free splittings from other $FZ_N^{max}$-maximal splittings.

\begin{figure}
\centering
\includegraphics{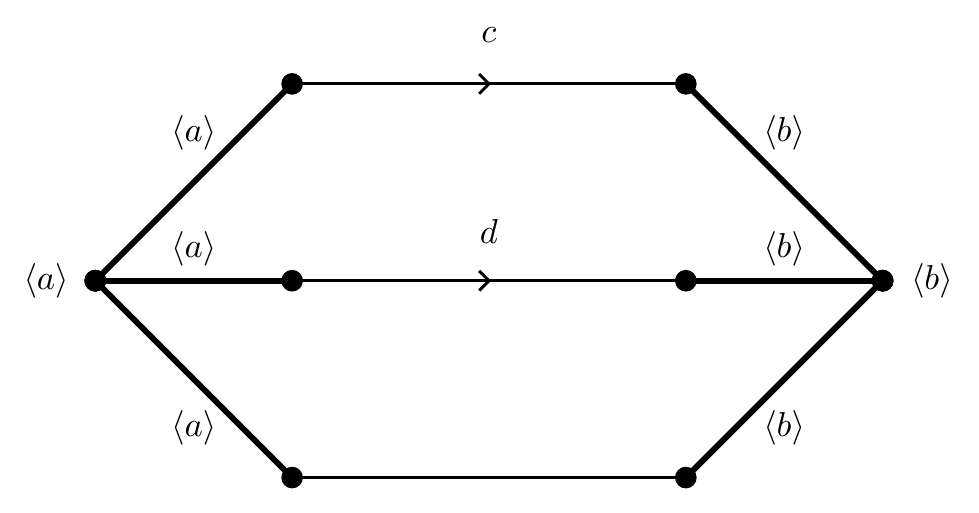}
\caption{An $FZ_4^{max}$-maximal splitting that fails to satisfy the conclusion of Proposition~\ref{max-vsn}.}
\label{fig-max-fz}

\end{figure}

\subsection{The case of $FZ_N^{max}$.}

We will need to make one more observation about maximal free splittings of $F_N$.

\begin{prop} \label{maximal-sphere-2}
Let $T$ be a maximal free splitting of $F_N$, and let $e_1$ and $e_2$ be two edges in $T$ that do not belong to the same $F_N$-orbit of edges. Assume that equivariantly collapsing either $e_1$ or $e_2$ to a point yields a splitting whose only adjacent $FZ_N^{max}$-maximal splitting is $T$. Then equivariantly collapsing both $e_1$ and $e_2$ to points yields a splitting of finite valence in $FZ_N^{max}$. 
\end{prop}

\begin{proof}
As noticed in the proof of Proposition \ref{maximal-sphere}, the edges $e_1$ and $e_2$ project to loop-edges in the quotient graph $T/F_N$. The claim follows by noticing that two such edges cannot be adjacent in $T/F_N$, as all vertices of a maximal free splitting have valence $3$. Hence the only way of attaching back edges to the splitting obtained by equivariantly collapsing both $e_1$ and $e_2$ to points is to attach $e_1$ and/or $e_2$.
\end{proof}

In contrast to free splittings, an $FZ_N^{max}$-maximal splitting which contains a nontrivial edge stabilizer fails to satisfy either Proposition \ref{maximal-sphere} or Proposition \ref{maximal-sphere-2}. 

\begin{prop}\label{maximal-torus} 
Let $T\in FZ_N^{max}$ be an $FZ_N^{max}$-maximal splitting that has some nontrivial edge stabilizer. At least one of the following two points holds:

\begin{enumerate}

\item There exists an edge $e$ in $T$ such that equivariantly collapsing $e$ to a point yields a splitting of infinite valence in $FZ_N^{max}$.

\item There exist two edges $e_1$ and $e_2$ in $T$ such that:

\begin{enumerate}
\item Equivariantly collapsing either $e_1$ or $e_2$ to a point yields a splitting whose only adjacent $FZ_N^{max}$-maximal splitting is $T$.
\item Equivariantly collapsing both $e_1$ and $e_2$ yields a splitting of infinite valence in $FZ_N^{max}$.\end{enumerate}
\end{enumerate}
\end{prop}

To prove this, we first take a detour to look at \emph{cylinders} in such splittings. A \emph{cylinder} in $T$ is a maximal subtree $Y$ in $T$ with the property that all edges of $Y$ have the same stabilizer. Note that an unfoldable edge (as defined in Section~\ref{s:vs}) is extremal in its associated cylinder (it is adjacent to a leaf vertex of $Y$).

\begin{lemma}\label{semi-interior}
Let $T\in FZ_N^{max}$ be a splitting that has some nontrivial edge stabilizer. There exists a cylinder $Y$ in $T$ with nontrivial  edge stabilizers that either contains an unfoldable edge which does not belong to any stabilized tripod, or contains two unfoldable edges in distinct $F_N$-orbits attached to the same vertex. \end{lemma}

\begin{proof} 
Suppose otherwise. By Lemma~\ref{unfold}, the tree  $T$ contains an unfoldable edge $e$.  Let $Y$ be the cylinder in $T$ containing $e$. The edge $e$ belongs to some stabilized tripod in $Y$, and no edge in $Y$ adjacent to $e$ is unfoldable.  Let $T'$ be the tree obtained from $T$ by fully unfolding $e$, and $f:T' \to T$ the associated fold map. Let $Y'$ be the cylinder of $T'$ corresponding to the preimage of $Y\setminus e$ under $f$.  Outside of the $F_N$--orbit of $Y'$, cylinders of $T'$ are mapped isomorphically under $f$ to cylinders of $T$. We claim that if an edge $e'$ with a nontrivial stabilizer in $T'$ is unfoldable then $f(e')$ is unfoldable in $T$. This claim, along with the fact that a cylinder of $T'$ is either isomorphic to $Y'$ or has the same structure as a cylinder of $T$, implies that $T'$ also does not satisfy the lemma (recall that no edge $e'$ adjacent to $e$ in $Y$ is unfoldable). We then obtain a contradiction by using induction on the number of $F_N$-orbits of edges with nontrivial stabilizers, as a splitting with only one such orbit certainly does satisfy the lemma. Our claim follows by looking at the peripheral subgroups of the vertex groups of $T'$ (i.e., the set of $G_v$-conjugacy classes of edge groups in a vertex group $G_v$), as the peripheral subgroups of $G_v$ (counted with multiplicity) determine which edges are unfoldable at $v$. 

Suppose that $e' \subseteq T'$ is unfoldable at a vertex $v \in T'$. Let $v_1$ be the vertex of $e$ in the interior of $Y$, and $v_2$ the vertex of $e$ which is a leaf of $Y$. If $f(v)$ is not in the orbit of either $v_1$ or $v_2$ then $G_v=G_{f(v)}$ and these groups have the same peripheral subgroups, with same multiplicity, hence $f(e')$ is unfoldable in $T$. If $f(v)=v_1$, then as $v_1$ belongs to a tripod in $Y$, the group $G_e$ is still a peripheral subgroup of $G_v=G_{v_1}$ with multiplicity at least $2$ and the same assertion holds. The remaining case to consider is when the terminal vertex of $e' \subseteq T'$ satisfies $f(v)=v_2$. As we have unfolded the edge $e$ at $v_2$ there is a free factor decomposition $G_{v_2}=G_e\ast G_v$, and as $e'$ is unfoldable in $T'$ there is a free factor decomposition $G_v=G_{e'} \ast A$, with $A$ containing a representative of $[G_{e''}]$ for each edge $e''$ adjacent to $v$ with $[e'']\neq[e']$. The free factor decomposition $G_{v_2}=G_e\ast G_{e'} \ast A$ then witnesses the fact that $f(e')$ is unfoldable in $T$. 
\end{proof}

\begin{proof}[Proof of Proposition \ref{maximal-torus}]
By Lemma~\ref{semi-interior} there exists an unfoldable leaf-edge $e$ of a cylinder in $T$ such that either $e$ does not belong to any stabilized tripod, or $e$ is adjacent to another unfoldable leaf-edge in the same cylinder. As $T$ is $FZ_N^{max}$-maximal, Lemma~\ref{unfold2} tells us that the edge $e$ has an adjacent vertex $v$ of valence $2$ (in the quotient graph $T/F_N$) such that its second adjacent edge $e'$ has a trivial stabilizer. Let $v'$ be the other vertex of $e$. 

If $e$ does not belong to any stabilized tripod, then the argument in the proof of Proposition \ref{max-vsn} applies to show that the first conclusion of the proposition holds.

If $e$ belongs to a stabilized tripod, then $e$ is adjacent to another unfoldable edge $e_2$ in the corresponding cylinder $Y$. If the rank of the stabilizer of the common vertex $v'$ of $e$ and $e_2$ is at least $2$, then equivariantly collapsing one of these edges to a point yields a splitting of infinite valence, as in the second paragraph of the proof of Proposition \ref{max-vsn}. Otherwise, we claim that the pair $\{e,e_2\}$ satisfies the second conclusion of the proposition. Indeed, as $e_2$ is unfoldable and $T$ is $FZ_N^{max}$-maximal, the edge $e_2$ is also adjacent to an edge with trivial stabilizer. Equivariantly collapsing both $e$ and $e_2$ to points yields a splitting of infinite valence by the same arguments as in the last paragraph of the proof of Proposition \ref{max-vsn}. Let $T'$ be a splitting obtained by collapsing only one of these orbits of edges. Maximality of $T$ implies that all possible refinements of $T'$ are obtained by attaching an edge at $v'$, and as $v'$ has cyclic stabilizer, the only way to do this yields the splitting $T$ back. 
\end{proof}

\begin{cor} \label{stable-subcomplex}
Any isometry of $FZ_N^{max}$ preserves $FS_N$ (setwise).
\end{cor}

\begin{proof}
We proceed in the same way as for $FZ_N$ and $VS_N$. Proposition~\ref{max-preserved} tells us that any isometry preserves the set of maximal splittings. Furthermore, Remark~\ref{edge-inv} tells us that the number of edges in a splitting and thus the property of adjacent vertices being collapses or refinements are also invariant under an isometry. Proposition~\ref{maximal-torus} tells us that every non-free $FZ_N^{max}$-maximal splitting $T$ either has a one-edge collapse with infinite valence in $FZ_N^{max}$, or a two-edge collapse with infinite valence such that both intermediate one-edge collapses have $T$ as their unique maximal refinement. Neither of these properties hold for a maximal free splitting (Propositions \ref{maximal-sphere} and \ref{maximal-sphere-2}). It follows that any isometry of $FZ_N^{max}$  preserves the set of maximal free splittings, and as free splittings are exactly the splittings of distance at most one from a maximal free splitting, such an isometry preserves $FS_N$ also.
\end{proof}

\section{Automorphisms of the free splitting graph} 

 Let $FS_N'$ be the graph whose vertices are one-edge free splittings of $F_N$, two vertices being joined by an edge if they admit a common refinement. If one adds higher-dimensional simplices to $FS_N'$ in a natural way then $FS_N$ becomes the 1-skeleton of the barycentric subdivision of $FS_N'$. The automorphism group of $FS_N'$ was computed by Aramayona and Souto in \cite{AS11}.

\begin{theo} (Aramayona-Souto \cite{AS11})\label{Aramayona-Souto}
For all $N\ge 3$, the natural map from $\text{Out}(F_N)$ to the isometry group of $FS'_N$ is an isomorphism. \end{theo} 

This implies a similar statement for $FS_N$. 

\begin{prop}\label{free-automorphisms}
For all $N\ge 3$, the natural map from $\text{Out}(F_N)$ to the isometry group of $FS_N$ is an isomorphism.
\end{prop}

\begin{proof}
Let $f$ be an isometry of $FS_N$, and denote by $X$ the subset of $FS_N$ consisting of one-edge splittings. Proposition \ref{max-preserved} implies that $f(X)=X$. In addition, two distinct one-edge splittings $T_1$ and $T_2$ admit a common refinement if and only if $d_{FS_N}(T_1,T_2)=2$. Hence $f$ maps pairs of compatible one-edge splittings to pairs of compatible one-edge splittings, so it induces an isometry of $FS'_N$. Using Theorem \ref{Aramayona-Souto}, we can thus assume, up to precomposing $f$ by an element of $\text{Out}(F_N)$, that $f$ fixes $X$ pointwise. As every free splitting is characterized by the set of one-edge splittings that are adjacent to it in $FS_N$ (Theorem \ref{t:ss}), this implies that $f$ is the identity map. The claim follows.
\end{proof}

As we now know that $FS_N$ is preserved by isometries of $FZ_N$, $FZ_N^{max}$, and $VS_N$, composing with an appropriate element of $\text{Out}(F_N)$ allows us to restrict our attention to isometries which fix $FS_N$ pointwise. 

\section{Any isometry that fixes $FS_N$ pointwise is the identity.}\label{sec-end}

Throughout the section, we will denote by $\mathcal{G}$ one of the graphs $FZ_N$, $FZ_N^{max}$ or $VS_N$. We define a \emph{bad} splitting of $F_N$ to be a splitting of the form $(F_{N-1}\ast\langle w^t\rangle)\ast_{\langle w\rangle}$, where $t\in F_N$ is a stable letter of the HNN extension, and $w\in F_{N-1}$ is not contained in any proper free factor of $F_{N-1}$. A \emph{good} splitting is a splitting which is not bad. The following proposition gives a characterization of bad splittings among one-edge splittings. Given $w\in F_N$, we denote by $\text{Fill}(w)$ the smallest free factor of $F_N$ that contains $w$.

\begin{prop} \label{characterization-bad-torus}
Let $T$ be a one-edge splitting of $F_N$. 
\begin{itemize} 
\item If $T$ is good, then there are infinitely many one-edge free splittings that are compatible with $T$.
\item If $T$ is bad, of the form $(A\ast\langle w^t\rangle)\ast_{\langle w\rangle}$, then the splitting $A\ast$ is the only one-edge free splitting that is compatible with $T$. 
\end{itemize}
\end{prop}

\begin{proof}
Let $T$ be a good one-edge splitting. The situation when $T$ is a free splitting or a separating one-edge $\mathbb{Z}$-splitting is covered in the proof of Lemma \ref{one-edge-valence}, so we assume that $T$ is of the form $(A\ast\langle w^t\rangle)\ast_{\langle w\rangle}$, where $A$ is a corank one free factor of $F_N$, and $w$ is contained in some proper free factor of $A$. Let $A'$ be a complementary free factor of $\text{Fill}(w)$ in $A$, and let $\{a'_1,\dots,a'_i\}$ be a free basis of $A'$. Then for all $k\in\mathbb{Z}$, the free splitting $(\text{Fill}(w)\ast\langle t\rangle)\ast\langle a'_1(w^t)^k,\dots,a'_i\rangle$ is compatible with $T$. The element $a'_1(w^t)^l$ is elliptic in this splitting if and only if $l=k$, so we obtain infinitely many distinct splittings as $k$ varies.

Now assume that $T$ is a bad one-edge splitting, of the form $(A\ast\langle w^t\rangle)\ast_{\langle w\rangle}$, where $A$ is a corank one free factor of $F_N$, and $w$ is not contained in any proper free factor of $A$. Let $S$ be a free splitting that is compatible with $T$. Then the common two-edge refinement $S'$ of $S$ and $T$ has one of the forms displayed on Figure \ref{fig-two-edge-refinement}, for some free factors $A'$ and $B'$ of $F_N$ and some $t',t_1,t_2\in F_N$. In Cases 1 and 2, slightly unfolding the edge with nontrivial stabilizer in $S'$ as in Lemma \ref{unfold} shows that $w$ should be contained in a corank $2$ free factor of $F_N$, a contradiction. Hence Case 3 occurs, and as $w\in A'$, we necessarily have $A'=\text{Fill}(w)$. The claim follows. 
\end{proof}

\begin{figure}
\begin{center}
\includegraphics{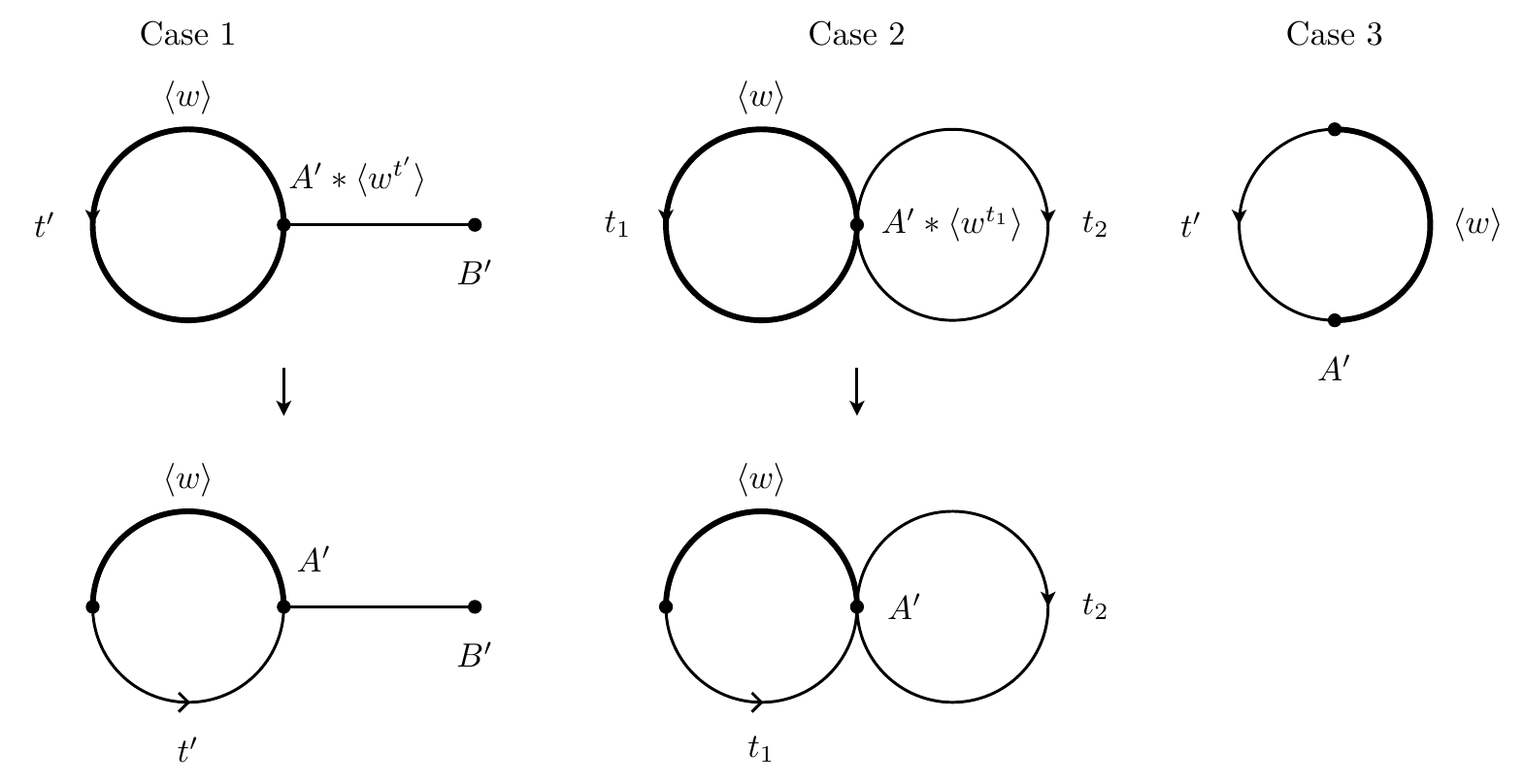}
\caption{The situation in the proof of Proposition \ref{characterization-bad-torus}.}
\label{fig-two-edge-refinement}
\end{center}
\end{figure}

From Proposition~\ref{max-preserved} we know that the set of one-edge splittings is preserved under any isometry of $\mathcal{G}$. Proposition \ref{characterization-bad-torus} then implies: 

\begin{prop} \label{preservation}
Any isometry of $\mathcal{G}$ which restricts to the identity on $FS_N$ preserves both the set of bad one-edge splittings and the set of good one-edge splittings (setwise).
\qed
\end{prop}

We will now prove that any such isometry fixes the set of good one-edge splittings pointwise, by showing that any good one-edge splitting is characterized by the collection of one-edge free splittings that are compatible with it. The following lemma can be proved by considering common refinements of $T$ and $T'$ on a case-by-case basis. 

\begin{lemma}\label{disjointness}
Let $T$ and $T'$ be two compatible cyclic splittings of $F_N$.
\begin{itemize}
\item Any element of $F_N$ that fixes an edge of $T$ is elliptic in $T'$.  
\item If $T$ and $T'$ are both separating one-edge splittings, then every elliptic subgroup of $T'$ is either contained in or contains an elliptic subgroup of $T$.
\item If $T$ is a one-edge nonseparating splitting and $T'$ is a one-edge separating splitting, then some vertex subgroup of $T'$ is contained in a vertex subgroup of $T$.
\qed
\end{itemize}
\end{lemma}

\begin{prop} \label{distinguishing-tori}
Let $T$ and $T'$ be two good one-edge $\mathbb{Z}$-splittings of $F_N$. There exists a one-edge free splitting of $F_N$ which is compatible with exactly one of the splittings $T$ or $T'$.
\end{prop}

\begin{proof}
Up to interchanging $T$ and $T'$, we can assume that either $T$ is separating, or both $T$ and $T'$ are nonseparating.
\\
\\
\noindent \textit{Case 1} : The splitting $T$ is separating, of the form $A\ast_{\langle w\rangle}(B\ast\langle w\rangle)$, where $A$ and $B$ are nontrivial proper free factors of $F_N$, and $w\in A$. We denote by $\{b_1,\dots,b_k\}$ a free basis of $B$. Let $\text{Fill}(w)$ be the smallest free factor of $A$ that contains $w$, and let $A'$ be a complementary free factor (possibly $\text{Fill}(w)=A$ and $A'=\{e\}$). Let $S$ be a free splitting of $F_N$ whose quotient graph is obtained by taking a rose corresponding to a basis of $A'$ and a rose corresponding to a basis of $B$ and attaching each rose by an edge with trivial stabilizer to a vertex with stabilizer $\text{Fill}(w)$. This is displayed in Figure \ref{fig-case1}. It is compatible with $T$. If $T'$ is not compatible with $S$, then we are done, so we assume that $T'$ is compatible with $S$. The splitting $T'$ is thus obtained from $S$ by first equivariantly inserting an edge $e$ with nontrivial stabilizer (which collapses to the only vertex with nontrivial stabilizer in $S$) to get a splitting $T''$, and then collapsing all edges but $e$ in $T''$. This amounts to splitting the vertex group $\text{Fill}(w)$ of $S$ over $\mathbb{Z}$, which means inserting in $S$ either a "horizontal" separating edge (Cases 1.1 to 1.4), a loop-edge (Case 1.5) or a "vertical" leaf  edge out of the graph (Cases 1.6 and 1.7). Thanks to Lemma \ref{unfold}, we know that the splitting we get has one of the following forms.
\\
\\
\noindent \textit{Case 1.1} : The splitting $T'$ is of the form $(A'\ast C)\ast_{\langle w'\rangle}[(D\ast\langle w'\rangle)\ast B]$, where $C$ and $D$ are complementary proper free factors of $\text{Fill}(w)$, and $w'\in C$.\\
Then $T'$ is compatible with $(A'\ast C)\ast (D\ast B)$. However, if $T$ were compatible with $(A'\ast C)\ast (D\ast B)$, then the first point of Lemma \ref{disjointness} implies that $w$ should be conjugated into either $A'\ast C$ or $D\ast B$, which is not the case.
\\
\\
\noindent \textit{Case 1.2} : The splitting $T'$ is of the form $[A'\ast (C\ast\langle w'\rangle)]\ast_{\langle w'\rangle}(D\ast B)$, where $C$ and $D$ are complementary proper free factors of $\text{Fill}(w)$, and $w'\in D$.\\
Then the same splitting as in Case 1.1 works. 
\\
\\
\textit{Case 1.3} : The splitting $T'$ is of the form $(A'\ast \text{Fill}(w))\ast_{\langle w'\rangle}(B\ast\langle w'\rangle)=A\ast_{\langle w'\rangle}(B\ast\langle w'\rangle)$, where $w'\in\text{Fill}(w)$ and $w'\neq w^{\pm 1}$.\\
Up to exchanging the roles of $T$ and $T'$, we can assume that $w\notin\langle w'\rangle$. Then $T$ is compatible with $A\ast \langle b_1w,b_2,\dots,b_k\rangle$. However, if $T'$ were compatible with $A\ast \langle b_1w,b_2,\dots,b_k\rangle$, then by the second point of Lemma \ref{disjointness}, the subgroup $\langle b_1w,b_2,\dots,b_k\rangle$ should either contain or be contained in $B\ast \langle w'\rangle$ (these vertex stabilizers are malnormal and have nontrivial intersection). This is not the case, as $w' \not \in \langle b_1w,b_2,\dots,b_k\rangle $ and $b_1w \not\in B\ast \langle w'\rangle$. 
\\
\\
\noindent \textit{Case 1.4}: The splitting $T'$ is of the form $(A'\ast \langle w'\rangle)\ast_{\langle w'\rangle}(\text{Fill}(w)\ast B)$, with $w'\in \text{Fill}(w)$.\\
If $\text{Fill}(w)=\langle w\rangle$, then $T'$ is of the form $(B\ast\langle w\rangle)\ast_{\langle w'\rangle}(A'\ast\langle w'\rangle)$ (while $T$ is of the form $(B\ast\langle w\rangle)\ast_{\langle w\rangle} (A'\ast\langle w\rangle)$), and the claim follows from the argument of Case 1.3. Otherwise, let $w''\in \text{Fill}(w)\smallsetminus\langle w\rangle$. Then $T'$ is compatible with $A\ast \langle b_1w'',b_2,\dots,b_k\rangle$ (they have a common two-edge refinement of the form $(A'\ast\langle w'\rangle)\ast_{\langle w'\rangle} \text{Fill}(w)\ast\langle b_1w'',b_2,\dots,b_k\rangle$), while $T$ is not by the same argument as in Case 1.3.
\\
\\
\noindent \textit{Case 1.5} : The splitting $T'$ is of the form $[A'\ast(C\ast\langle w'\rangle^t)\ast B]\ast_{\langle w'\rangle}$, where $C$ is a corank one free factor of $\text{Fill}(w)$, and $w'\in C$, and $t\in\text{Fill}(w)$.\\
Then the splitting $(A'\ast B\ast C)\ast$ is compatible with $T'$, but not with $T$ since $w$ is not conjugated into $A'\ast B\ast C$.
\\
\\
\textit{Case 1.6} : The splitting $T'$ is of the form $(A'\ast B\ast C)\ast_{\langle w'\rangle} (C'\ast\langle w'\rangle)$, where $C\ast C'=\text{Fill}(w)$, and $w'\in C$.\\
In particular, we have $C'\neq\{e\}$, and the splitting $(A'\ast B\ast C)\ast C'$ is compatible with $T'$, but not with $T$ since $w$ is neither conjugated into $A'\ast B\ast C$, nor into $C'$.
\\
\\
\textit{Case 1.7}: The splitting $T'$ is of the form $[A'\ast B\ast (C\ast\langle w'\rangle)]\ast_{\langle w'\rangle} C'$, where $C\ast C'=\text{Fill}(w)$, and $w'\in C'$.\\
If $A'=\{e\}$ and $C'=\text{Fill}(w)$, then $T'$ is of the form $\text{Fill}(w)\ast_{\langle w'\rangle} (B\ast \langle w'\rangle)$ (while $T$ is of the form $\text{Fill}(w)\ast_{\langle w\rangle} (B\ast\langle w\rangle)$), so the claim follows from the argument of Case 1.3. 

Otherwise, the same splitting as in Case 1.6 works by the second point of Lemma \ref{disjointness}. Indeed, in this case, we have $B\neq\{e\}$ and $C\neq\text{Fill}(w)$. The group $A=A'\ast\text{Fill}(w)$ is neither contained in nor contains $A'\ast B\ast C$, so the conclusion follows.
\\
\\
\noindent\textit{Case 2} : Both splittings $T$ and $T'$ are nonseparating. By Lemma \ref{unfold}, the splitting $T$ is of the form $(A\ast\langle w^t\rangle)\ast_{\langle w\rangle}$ for some corank one free factor $A$ of $F_N$ and some $w\in A$. We denote by $\text{Fill}(w)$ the smallest free factor of $A$ that contains $w$, and we let $A'$ be a complementary free factor in $A$. As $T$ is assumed to be good, we have $\text{Fill}(w)\neq A$, and hence $A'\neq\{e\}$. We denote by $\{a'_1,\dots,a'_k\}$ a free basis of $A'$. Let $S$ be the free splitting displayed on Figure \ref{fig-case2}, it is compatible with $T$. Again, we may assume that $T'$ is compatible with $S$ (otherwise we are done). Passing from $S$ to $T'$ again requires inserting an edge with $\mathbb{Z}$ stabilizer to get a new tree $T''$, then equivariantly collapsing its complement in $T''$. As $T'$ is nonseparating, the inserted edge can either be a separating edge lying on the loop labelled by $t$ in $S$ (which leads to Cases 2.1 to 2.4) or a loop-edge (Case 2.5). 
\\
\\
\textit{Case 2.1}: The splitting $T'$ is of the form $[A'\ast C\ast (D\ast\langle w'\rangle)^{t}]\ast_{\langle w'\rangle}$, where $C$ and $D$ are two complementary free factors of $\text{Fill}(w)$, and $w'\in C$. \\
We have $w'\neq w^{\pm1}$ (otherwise $C=\text{Fill}(w)$ and $T'=T$). Up to exchanging the roles of $T$ and $T'$, we can assume that $w\notin\langle w'\rangle$. In this case, the splitting $T'':=(\text{Fill}(w)\ast\langle t\rangle)\ast\langle a'_1w^t,a'_2,\dots,a'_k\rangle$ is compatible with $T$. However, if $T''$ were also compatible with $T'$, by the third point of Lemma \ref{disjointness} one of the vertex stabilizers of $T''$ should be elliptic in $T'$. This is not the case, since both $t$ and $a'_1w^t$ are hyperbolic in $T'$.
\\
\\
\textit{Case 2.2} : The splitting $T'$ is of the form $[A'\ast C\ast (D\ast\langle w'\rangle)^{t^{-1}}]\ast_{\langle w'\rangle}$, where $C$ and $D$ are two complementary free factors of $\text{Fill}(w)$, and $w'\in C$.\\
Then the same splitting as in Case 2.1 works.
\\
\\
\textit{Case 2.3} : The splitting $T'$ is of the form $[A'\ast (C\ast\langle w'\rangle)\ast D^t]\ast_{\langle w'\rangle}$, where $C$ and $D$ are two complementary free factors of $\text{Fill}(w)$, and $w'\in D$.\\
We may also assume $C\ast\langle w'\rangle\neq\text{Fill}(w)$, otherwise we are in a particular case of Case 2.1. Let $w''\in \text{Fill}(w)\smallsetminus (C\ast\langle w'\rangle)$. Then the splitting $T'':=(\text{Fill}(w)\ast\langle t\rangle)\ast\langle a'_1w'',a'_2,\dots,a'_k\rangle$ is compatible with $T$. If it were also compatible with $T'$, the third point of Lemma \ref{disjointness} would imply that one of the edge groups of $T''$ is elliptic in $T'$. This is not the case, since both $t$ and $a'_1w''$ are seen to be hyperbolic in $T'$.
\\
\\
\textit{Case 2.4} : The splitting $T'$ is of the form $[A'\ast (C\ast\langle w'\rangle)\ast D^{t^{-1}}]\ast_{\langle w'\rangle}$, where $C$ and $D$ are two complementary free factors of $\text{Fill}(w)$, and $w'\in D$.\\
The same splitting as in Case 2.3 works.
\\
\\
\textit{Case 2.5} : The splitting $T'$ is of the form $[A'\ast(C\ast\langle t\rangle)\ast\langle w'^{t'}\rangle]\ast_{\langle w'\rangle}$, where $C$ is a corank one free factor of $\text{Fill}(w)$, and $w'\in C$, and $t'\in\text{Fill}(w)$.\\
Let $w''\in\text{Fill}(w)\smallsetminus (C\ast\langle w'^{t'}\rangle)$. Again, the splitting $(\text{Fill}(w)\ast\langle t\rangle) \ast \langle a'_1w'',\dots,a'_k\rangle$ is compatible with $T$, but not with $T'$ since both $w''$ and $a'_1w''$ are hyperbolic in $T'$. 
\end{proof}

\begin{figure}
\begin{center}
\includegraphics{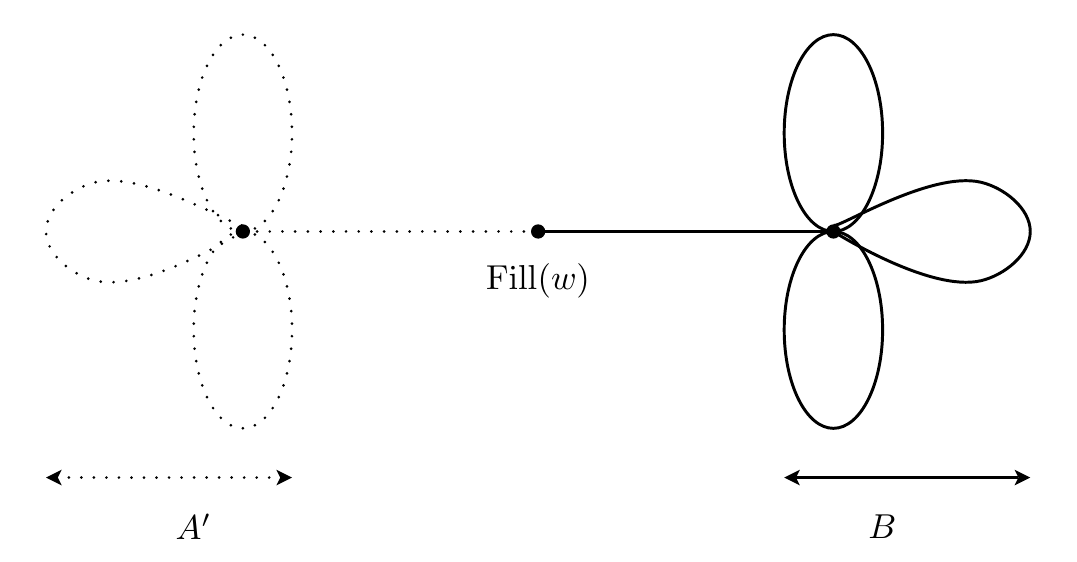}
\caption{The splitting $S$ in Case 1 of the proof of Proposition \ref{distinguishing-tori}.}
\label{fig-case1}
\end{center}
\end{figure}

\begin{figure}
\begin{center}
\includegraphics{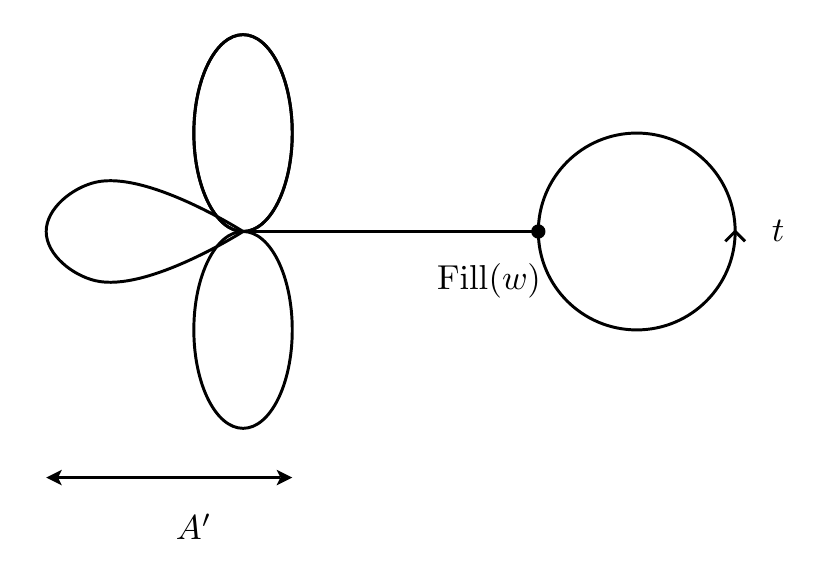}
\caption{The splitting $S$ in Case 2 of the proof of Proposition \ref{distinguishing-tori}.}
\label{fig-case2}
\end{center}
\end{figure}

\begin{prop} \label{fix-good}
Any isometry of $\mathcal{G}$ that restricts to the identity on $FS_N$ fixes the set of good one-edge splittings pointwise.
\end{prop}

\begin{proof}
By Proposition \ref{preservation}, any isometry $f$ of $\mathcal{G}$ that restricts to the identity on $FS_N$ preserves the set $X$ of good one-edge splittings. If $T$ and $T'$ are two distinct good one-edge splittings, Proposition \ref{distinguishing-tori} implies that the set of one-edge free splittings at distance at most $2$ from $T$ and the set of one-edge free splittings at distance at most $2$ from $T'$ are distinct. Hence $f$ fixes $X$ pointwise.
\end{proof}

\begin{prop} \label{distinguish-bad}
Let $T$ and $T'$ be two distinct bad one-edge cyclic splittings of $F_N$. Then there exists a good one-edge cyclic splitting $T''$ of $F_N$ which is compatible with exactly one of the splittings $T$ and $T'$. If in addition $T$ and $T'$ are maximally-cyclic, then $T''$ can be chosen to be maximally-cyclic. 
\end{prop}

\begin{proof}
Let $T$ and $T'$ be two bad one-edge splittings. Proposition \ref{characterization-bad-torus} implies that there is exactly one free splitting $S$ (resp. $S'$) which is compatible with $T$ (resp. $T'$). If $S\neq S'$, then we are done, so we assume that $S=S'$. The splitting $T$ is of the form $(A\ast\langle w^{t}\rangle)\ast_{\langle w\rangle}$ for some corank one free factor $A$ of $F_N$ and some $w\in A$. 
\\
\\
\textit{Case 1} : The splitting $T'$ is of the form $(A\ast\langle w'^t\rangle)\ast_{\langle w'\rangle}$.\\
Up to exchanging the roles of $T$ and $T'$, we can assume that $w'\notin\langle w\rangle$. Denoting by $T''$ the separating (and hence good) splitting $A\ast_{\langle w\rangle}\langle w,t\rangle$, we get that $T''$ is compatible with $T$, but not with $T'$. Indeed, if it were, then in a common refinement of $T'$ and $T''$, the axis of $t$ should meet the fixed point set of $w'$, which leads to a contradiction because $w'$ and $t$ belong to distinct elliptic subgroups of $T''$. 
\\
\\
\textit{Case 2} : The splitting $T'$ is of the form $(A\ast\langle w'^{t^{-1}}\rangle)\ast_{\langle w'\rangle}$.\\
Let $w''\in A$ be contained in some proper free factor of $A$ (this exists because $N\ge 3$). The splitting $(A\ast\langle w''^{t^{-1}}\rangle)\ast_{\langle w''\rangle}$ is then a good splitting which is compatible with $T$, but not with $T'$.  (This can be seen by looking at the possible forms of a two-edge refinement of $T'$ with this splitting, and reaching a contradiction in each case.)
\end{proof}

\begin{prop} \label{fix-bad}
Any isometry of $\mathcal{G}$ that restricts to the identity on $FS_N$ fixes the set of bad one-edge splittings of $F_N$ pointwise.
\end{prop}

\begin{proof}
Let $f$ be an isometry of $\mathcal{G}$ that restricts to the identity on $FS_N$. It follows from Proposition \ref{preservation} that $f$ preserves the set of bad one-edge splittings of $F_N$. If $T$ and $T'$ are two bad one-edge splittings in $\mathcal{G}$, then Proposition \ref{distinguish-bad} implies that the set of good one-edge splittings at distance $2$ from $S$ in $\mathcal{G}$ is distinct from the set of good one-edge splittings at distance $2$ from $T'$ in $\mathcal{G}$. The claim thus follows from Proposition \ref{fix-good}.
\end{proof}

\begin{proof}[Proof of Theorem~A]
Let $\mathcal{G}$ be one of the graphs $FZ_N$, $FZ_N^{max}$ or $VS_N$, and let $f$ be an isometry of $\mathcal{G}$. Corollary \ref{fz-fs} (for the case where $\mathcal{G}=FZ_N$), Corollary \ref{vs-fs} (when $\mathcal{G}=VS_N$) or Corollary \ref{stable-subcomplex} (when $\mathcal{G}=FZ_N^{max}$) show that $f$ preserves $FS_N$ setwise. Proposition \ref{free-automorphisms} then implies that up to composing $f$ with an element of $\text{Out}(F_N)$, we may assume that $f$ restricts to the identity on $FS_N$, and we want to show that $f=\text{id}$. The set $X_1$ (resp. $X_2$) of good (resp. bad) one-edge splittings in $\mathcal{G}$ is fixed pointwise by $f$ (Propositions \ref{fix-good} and \ref{fix-bad}). Using Theorem \ref{t:ss}, we get that for all $T,T'\in \mathcal{G}\smallsetminus (X_1\cup X_2)$, the set of elements in $X_1\cup X_2$ at distance $1$ from $T$ differs from the set of elements at distance $1$ from $T'$, so $f$ also fixes $\mathcal{G}\smallsetminus (X_1\cup X_2)$ pointwise. Theorem~A follows.
\end{proof}

\section{Variations over complexes}

Denote by $(FZ_N)'$ (respectively $(FZ_N^{max})'$) the graph whose vertices are the equivalence classes of one-edge cyclic (resp. maximally-cyclic) splittings of $F_N$, in which two vertices are joined by an edge whenever they are compatible. These are dual versions of the graphs we have been dealing with so far.

\begin{prop}\label{automorphisms-vs'}
For all $N\ge 3$, the natural maps from $\text{Out}(F_N)$ to the isometry groups of $(FZ_N)'$ and $(FZ_N^{max})'$ are isomorphisms.
\end{prop}

\begin{proof}
Let $\mathcal{G}$ be one of the graphs $FZ_N$ or $FZ_N^{max}$, and $\mathcal{G}'$ be its "dual" version. Let $f$ be an isometry of $\mathcal{G}'$. Let $T\in \mathcal{G}$. The set of one-edge splittings at distance at most $1$ from $T$ in $\mathcal{G}$ is a collection of pairwise compatible one-edge splittings, hence it is mapped by $f$ to a collection $\mathcal{C}$ of pairwise compatible one-edge splittings in $\mathcal{G}$. Thanks to Theorem \ref{t:ss}, we know that the splittings in $\mathcal{C}$ have a common refinement in $\mathcal{G}$, which is the unique tree $T'\in \mathcal{G}$ such that the set of one-edge splittings at distance at most $1$ from $T'$ is equal to $\mathcal{C}$. Hence $f$ induces a map $f':\mathcal{G}\to \mathcal{G}$, mapping $T$ to $T'$. This is an isometry: it is a bijection because $f$ is, and any collapse (resp. refinement) of a splitting is mapped by $f$ to a collapse (resp. refinement) of its $f$-image. By Theorem~A, there exists $\Phi\in\text{Out}(F_N)$ such that for all $T\in \mathcal{G}$, we have $f'(T)=T \cdot \Phi$. This holds in particular for one-edge splittings of $F_N$, for which $f(T)=f'(T)$. Hence $f$ is induced by an element of $\text{Out}(F_N)$.
\end{proof}

\begin{rk}
One might turn the graph $(FZ_N^{max})'$ into a higher-dimensional complex $\widehat{FZ_N^{max}}$ (or $\widehat{VS_N}$) by adding a $k$-simplex for each maximally-cyclic (or very small) splitting of $F_N$ having exactly $k$ orbits of edges. Thanks to Theorem \ref{bf}, we know that the complexes $\widehat{FZ_N^{max}}$ and $\widehat{VS_N}$ are finite-dimensional (whereas the same construction would lead to an infinite-dimensional complex in the case of general cyclic splittings of $F_N$). Any simplicial automorphism of either $\widehat{FZ_N^{max}}$ or $\widehat{VS_N}$ restricts to an isometry of its one-skeleton, hence the groups of simplicial automorphisms of $\widehat{FZ_N^{max}}$ and $\widehat{VS_N}$ also coincide with $\text{Out}(F_N)$. The complex $\widehat{FZ_N^{max}}$ is flag by Theorem \ref{t:ss}, while $\widehat{VS_N}$ is not.
\end{rk}

\bibliographystyle{amsplain}
\bibliography{Automorphisms-Bibliography}

\end{document}